\documentclass[a4paper,12pt]{amsart}
\usepackage{amsthm, amsmath,amssymb,amsxtra}

\newtheorem{theorem}{Theorem}[section]

\newtheorem{lemma}[theorem]{Lemma}

\theoremstyle{definition}

\newtheorem{remark}[theorem]{Remark}

\newtheorem*{theoremA}{Theorem A}
\newtheorem*{theoremB}{Theorem B}
\newtheorem*{theoremC}{Theorem C}

\newtheorem*{lemmaA1}{Lemma A.1}
\newtheorem*{lemmaA2}{Lemma A.2}
\newtheorem*{lemmaA3}{Lemma A.3}

\newcommand{\na}{\mathbb{N}}
\newcommand{\re}{\mathbb{R}}
\def\bra#1{\langle {#1} \rangle }
\def\r2n{\re^n\times (\re^n\!\setminus\!\{ 0\})}
\def\dbar{\mbox{\setbox0=\hbox{$d$}$d$\kern-.55\wd0\vbox{%
\hrule height.1ex width.75\wd0\kern1.3ex}}}

\begin{document}


\date{December 23, 2012}
\title[]{Estimates on Modulation Spaces for Schr\"odinger Evolution
Operators with Quadratic and  Sub-quadratic Potentials}

\author{Keiichi Kato, Masaharu Kobayashi and Shingo Ito}

\address{ 
Department of Mathematics \endgraf
Tokyo University of Science \endgraf
Kagurazaka 1-3, Shinjuku-ku, Tokyo 162-8601,
Japan
}
\email{kato@ma.kagu.tus.ac.jp}

\address{ 
Department of Mathematical Sciences, Faculty of Science \endgraf
Yamagata University \endgraf
Kojirakawa 1-4-12, Yamagata-City,
Yamagata 990-8560,
Japan\endgraf
}
\email{kobayashi@sci.kj.yamagata-u.ac.jp }

\address{ 
Department of Mathematics \endgraf
Tokyo University of Science \endgraf
Kagurazaka 1-3, Shinjuku-ku, Tokyo 162-8601,
Japan
}
\email{ito@ma.kagu.tus.ac.jp}

\maketitle

\begin{abstract}
In this paper we give new estimates for the solution to the
 Schr\"odinger equation with quadratic and sub-quadratic potentials
 in the framework of modulation spaces.
\end{abstract}


\maketitle

\section{Introduction}
In this paper, we shall give estimates for the solution to the time
 dependent Schr\"odinger equation
    \begin{align}
    \label{SE}
    \begin{cases}
    i\partial_t u(t,x)=-\frac{1}{2}\Delta u(t,x)+V(t,x)u(t,x),  &(t,x)\in
    \re\times\re^n,\\
    u(0,x)=u_0(x), &x\in \re^n
    \end{cases}
    \end{align}
 in the framework of modulation spaces. Here $i= \sqrt{-1}$,  $u(t,x)$
 is a complex valued function of $(t,x) \in {\mathbb R} \times {\mathbb
 R}^n$, $V(t,x)$ is a real valued function of $(t,x) \in {\mathbb R}
 \times {\mathbb R}^n$, $u_0(x)$ is a complex valued function of $x \in
 {\mathbb R}^n$, $\partial_t u= {\partial u}/{\partial t}$ and 
 $\Delta u= \sum^{n}_{i=1} {\partial^2 u }/{ \partial x_i^2}$.

We shall highlight the case $V \in C^{\infty}({\mathbb R} \times
 {\mathbb R}^n)$ and for all multi-indices $\alpha$ with $|\alpha|\ge 2$
 or $|\alpha|\ge 1$ there exists $C_\alpha >0$ such that 
    \begin{align}
    \label{assumption_V}
    |\partial^\alpha_xV(t,x)|\le C_{\alpha},
    \quad (t,x)\in \re\times\re^n .
    \end{align}

There are a large number of works devoted to study the equation \eqref{SE}.
Particularly, in the context of modulation spaces $M^{p,q}$,
 these types of issues were initiated in the works of 
 B\'{e}nyi-Gr\"{o}chenig-Okoudjou-Rogers \cite{Benyi et all},
 Wang-Hudzik \cite{Wang-Hudzik} and Wang-Zhao-Guo \cite{WZG}. 

\begin{theoremA}
(B\'{e}nyi-Gr\"{o}chenig-Okoudjou-Rogers \cite{Benyi et all}) 
Let $1 \leq p,q \leq \infty$ and 
 $\varphi_0 \in {\mathcal S}({\mathbf R}^n) \setminus \{ 0 \}$.
Suppose $V(t,x)=0$. 
Then there exists a positive constant $C$ such that 
    \begin{align*}
    \| u(t,\cdot)\|_{M^{p,q}_{\varphi_0}}\le C (1+|t|)^{n/2}
    \| u_0\|_{M^{p,q}_{\varphi_0}}, 
    \quad u_0\in\mathcal{S}(\re^n)
    \end{align*}
 for all $t\in\re$, where $u(t,x)$ is the solution of \eqref{SE} with 
 $u(0,x)=u_0(x)$. 
\end{theoremA}

\begin{theoremB}
(Wang-Hudzik \cite{Wang-Hudzik}) 
Let $2 \leq p \leq \infty$, 
 $1\le q\le \infty$, $1/p+1/p'=1$ and 
 $\varphi_0 \in {\mathcal S}({\mathbf R}^n) \setminus \{ 0 \}$.
Suppose $V(t,x)=0$.
Then  there exists positive constants $C$ and $C'$ such that 
    \begin{align*}
    \| u(t,\cdot)\|_{M^{p,q}_{\varphi_0}}\le C (1+|t|)^{-n(1/2-1/p)}
    \| u_0\|_{M^{p',q}_{\varphi_0}},
    \quad u_0\in\mathcal{S}(\re^n)
    \end{align*}
 and 
    \begin{align*}
    \| u(t,\cdot)\|_{M^{p,q}_{\varphi_0}}\le C' (1+|t|)^{n(1/2-1/p)}
    \| u_0\|_{M^{p,q}_{\varphi_0}},
    \quad u_0\in\mathcal{S}(\re^n)
    \end{align*}
 for all $t\in\re$, where $u(t,x)$ is the solution of \eqref{SE} with 
 $u(0,x)=u_0(x)$. 
\end{theoremB}

The studies of this theme have been developed by a number of authors
 using a large variety of methods (see, for example,  B\'{e}nyi-Okoudjou
 \cite{Benyi-Okoudjou},
 Cordero-Nicola \cite{Cordero Nicola JFA 2008}, 
\cite{Cordero Nicola MN 2008},
 Kobayashi-Sugimoto \cite{KS}, Miyachi-Nicola-Rivetti-Tabacco-Tomita
 \cite{Miyachi-NicolaRivetti-Tabacco-Tomita}, Tomita \cite{Tomita}, 
 Wang-Huang \cite{Wang Huang}). 

In our previous papers, we have the following estimates.

\begin{theoremC}
\label{free}
(Kato-Kobayashi-Ito \cite{K-K-I 1}, \cite{K-K-I 3}, \cite{K-K-I 2})
Let $1 \leq p,q \leq \infty$
 and $\varphi_0 \in {\mathcal S}({\mathbf R}^n) \setminus \{ 0 \}$.
   \begin{enumerate}
   \item [$(i)$] Suppose $V(t,x)=0$. Then
      \begin{align*}
      \| u(t, \cdot) \|_{M^{p,q}_{\varphi(t,\cdot)}}
      = \| u_0 \|_{M^{p,q}_{\varphi_0}},
      \quad u_0 \in {\mathcal S}({\mathbb R}^n)
      \end{align*}    
    holds for all $t\in\re$.
   \item[$(ii)$] Suppose $V(t,x)= \pm \frac{1}{2}|x|^2$. Then
      \begin{align*}
      \| u(t,\cdot) \|_{M^{p,p}_{\varphi(t,\cdot)}}= \| u_0 
      \|_{M^{p,p}_{\varphi_0}},\quad u_0 \in {\mathcal S}({\mathbb R}^n)
      \end{align*}
    holds for all $t\in\re$. 
   \end{enumerate}
In $(i)$ and $(ii)$, $u(t,x)$ and $\varphi(t,x)$ denote the solutions 
 of \eqref{SE} with $u(0,x)=u_0(x)$ and $\varphi(0,x)=\varphi_0(x)$.
\end{theoremC}

We remark that Theorem C covers Theorem A and B (see \cite{K-K-I 1}).

To state our results, we define the Schr\"odinger operator of a free
 particle  $e^{\frac{1}{2}it\Delta}$ by 
    $$(e^{\frac{1}{2}it\Delta}f)(x)=
    \mathcal{F}^{-1}_{\xi\to x}[e^{-\frac{1}{2}it|\xi|^2}
    \mathcal{F}f(\xi)](x), \quad f\in\mathcal{S}(\re^n).$$ 
Here we use the notation
 $\mathcal{F}f(\xi)=\int_{\re^n}f(x)e^{-ix\cdot\xi}dx$ for the Fourier
 transform of $f$ and 
 $\mathcal{F}^{-1}f(x)=\int_{\re^n} f(\xi)e^{ix\cdot\xi}\dbar\xi$ with
 $\dbar\xi =(2\pi)^{-n}d\xi$ for the inverse Fourier transform of $f$.
The following theorems are our main results.

\begin{theorem}
\label{p-p estimate}
Let $1\le p \le \infty$, 
 $\varphi_0\in\mathcal{S}(\re^n)\backslash\{ 0\}$ and $T>0$.
Set $\varphi (t,x)=e^{\frac{1}{2}it\Delta} \varphi_0(x)$.
If $V\in C^{\infty} (\re\times\re^{n})$ satisfies \eqref{assumption_V}
 for all multi-indices $\alpha$ with $|\alpha |\ge 2$, then 
 there exists $C_{T}>0$ such that 
    \begin{align*}
    \| u(t,\cdot) \|_{M^{p,p}_{\varphi (t,\cdot)}}
    \le C_{T} \|u_0\|_{M^{p,p}_{\varphi_0}},
    \quad u_0\in\mathcal{S}(\re^n)
    \end{align*}
 for all $t\in [-T,T]$, where $u(t,x)$ denotes the solution of
 \eqref{SE} in $C(\re ;L^2(\re^n))$
 with $u(0,x)=u_0(x)$.
\end{theorem}

In the above theorem, we cannot expect to replace the $M^{p,p}$
 norm with the $M^{p,q}$ norm.
In fact, when $V(t,x)=\frac{1}{2}|x|^2$ we have
    \begin{align*}
    \| u(t,\cdot)\|_{M^{p,q}_{\varphi (t,\cdot)}}=
    \| \|W_{\varphi_0} u_0(x\cos t-\xi\sin t, x\sin t+\xi \cos t)
    \|_{L^p_x}\|_{L^q_\xi}
    \end{align*}
 and thus
    \begin{align*}
    \left\| u\left(\frac{\pi}{2},\cdot\right)
    \right\|_{M^{p,q}_{\varphi (\frac{\pi}{2},\cdot)}}=
    \| \|W_{\varphi_0} u_0(\xi, x)\|_{L^p_x}\|_{L^q_\xi}
    \end{align*}
(refer to \cite{K-K-I 3}), but $\| \|W_{\varphi_0} u_0(\xi, x)
\|_{L^p_x}\|_{L^q_\xi}\le C \| \|W_{\varphi_0} u_0(x, \xi)
\|_{L^p_x}\|_{L^q_\xi}$ does not hold generally.
However, if we strengthen the assumption of $V$ then we can replace the 
 $M^{p,p}$ norm with $M^{p,q}$ norm.

\begin{theorem}
\label{p-q estimate}
Let $1\le p,q\le \infty$, 
 $\varphi_0\in\mathcal{S}(\re^n)\backslash\{ 0\}$ and  $T>0$.
Set $\varphi (t,x)=e^{\frac{1}{2}it\Delta} \varphi_0(x)$.
If $V\in C^\infty (\re\times \re^n)$ satisfies \eqref{assumption_V} 
 for all multi-indices $\alpha$ with $|\alpha|\ge 1$,
 then there exists $C_{T}>0$ such that 
    \begin{align*}
    \| u(t,\cdot) \|_{M^{p,q}_{\varphi (t,\cdot)}}
    \le C_{T} \|u_0\|_{M^{p,q}_{\varphi_0}},
    \quad u_0\in \mathcal{S}(\re^n)
    \end{align*}
 for all $t\in [-T,T]$, where $u(t,x)$ denotes the solution of
 \eqref{SE} in $C(\re ;L^2(\re^n))$ with $u(0,x)=u_0(x)$.
\end{theorem}

\begin{remark}
In Theorem \ref{p-p estimate} and Theorem \ref{p-q estimate},
 we assume $V(t,x)\in C^\infty (\re\times \re^n)$ but, in fact, it is
 enough to assume $V(t,x)$ is $C^2$-function in $t$ and 
 $C^{2[\frac{n}{2}]+4}$-function in $x$.
\end{remark}

This paper is organized as follows.
In Section 2, we give some notations and recall the definitions and 
 basic properties of wave packet transform and modulation spaces.
In Section 3, we give some properties concerning the orbit of the classical
 mechanics corresponding to the Schr\"odinger equation \eqref{SE}.
In Section 4, we prove Theorem \ref{p-p estimate}.
Finally, in Section 5, we prove Theorem \ref{p-q estimate}.

\section{Preliminaries}\label{Preliminaries}

\subsection{Notations}
For $x=(x_1,\ldots ,x_n)\in\re^n$ and a $m\times n$ matrix $A=(a_{ij})$,
 we denote
$$\bra{x}=(1+|x|^2)^{\frac{1}{2}},\quad
\|x\|_\infty =\underset{1\le j\le n}{\rm max}|x_j| \quad
{\text{and}}\quad 
 \|A\|_\infty=\underset{1\le j\le m,1\le k\le n}{\rm max}|a_{jk}|.$$
For a real valued function $V\in C^1(\re\times \re^n)$,
we put
$$\nabla_x V(t,x_1,\ldots ,x_n)=(\partial_{x_1}V(t,x_1,\ldots ,x_n)
 ,\ldots , \partial_{x_n}V(t,x_1,\ldots ,x_n) ).$$
Throughout this paper the letter $C$
denotes a constant, which may be different in each occasion.

\subsection{Wave Packet Transform}\label{Wave Packet Transform}
We recall the definition of the wave packet transform which is defined
 by C\'ordoba-Fefferman \cite{C-F-1}. Wave packet transform is called
 short time Fourier transform or windowed Fourier transform in several
 literatures. Let $f \in {\mathcal S}^\prime({\mathbb R}^n)$ and 
 $\varphi \in {\mathcal S}({\mathbb R}^n)\backslash \{0\}$. Then the
 wave packet transform $W_\varphi f(x,\xi)$ of $f$ with the wave packet
 generated by a function $\varphi$ is defined by 
    \begin{equation*}
    W_\varphi f(x,\xi) = \int_{{\mathbb R}^n} 
    \overline{ \varphi(y-x)} f(y) e^{-i y \cdot \xi} dy.
    \end{equation*} 
We call such $\varphi$ window function. Let $F$ be a function on
 ${\mathbb R}^n \times {\mathbb R}^n$. Then the (informal) adjoint
 operator $W^{*}_\varphi$ of $W_\varphi$ is defined by
    $$W^{*}_\varphi F(x)=\iint_{{\mathbb R}^{2n}} F(y,\xi)
    \varphi(x-y)  e^{ix \cdot \xi} dy \dbar\xi$$
with $\dbar \xi = (2 \pi)^{-n} d \xi$. It is known that for $\varphi,
 \psi \in {\mathcal S}({\mathbb R}^n)\backslash\{ 0\}$ satisfying
 $\langle \psi,\varphi \rangle \not=0$, we have the inversion formula
    \begin{equation*}
    \frac{1}{\langle \psi,\varphi \rangle} W_\psi^{*} W_\varphi f = f,
    \quad f \in {\mathcal S}^\prime({\mathbb R}^n) \label{inversion formula}
    \end{equation*}
 (\cite[Corollary 11.2.7]{Grochenig 2001}).

For the sake of convenience, we use the following notation
    \begin{align*}
    W_{\varphi(t,\cdot)}u(t,x,\xi)
    =W_{\varphi(t,\cdot)}[u(t,\cdot)](x,\xi)
    =\int_{\re^n}\overline{\varphi (t,y-x)}u(t,y)e^{-iy\cdot \xi}dy,
    \end{align*}
 where  
 $\varphi (t,x)$ and $u(t,x)$ are functions on  $\re\times\re^n$.

\subsection{Modulation Spaces}\label{MODULATION SPACES}
We recall the definition of modulation spaces $M^{p,q}$.
Let $1 \leq p,q \leq \infty$ and 
 $\varphi \in {\mathcal S}({\mathbb R}^n) \setminus \{ 0 \}$.
Then  the modulation space $M^{p,q}_\varphi({\mathbb R}^n)=M^{p,q}$
 consists of all tempered distributions 
 $f \in {\mathcal S}^\prime({\mathbb R}^n)$ such that the norm
    $$\| f \|_{M^{p,q}_\varphi} =
    \left( \int_{{\mathbb R}^n} \left( \int_{{\mathbb R}^n} |W_\varphi 
    f (x,\xi)|^p  dx  \right)^{q/p} d \xi  \right)^{1/q}
    = \|\| W_\varphi f(x,\xi) \|_{L^p_x}\|_{L^q_\xi}$$
 is finite (with usual modifications  if $p=\infty$ or $q=\infty$).

The space $M^{p,q}_\varphi({\mathbb R}^n)$ is a Banach space,
whose definition is  independent of the choice of the window function
$\varphi$, i.e., $M^{p,q}_\varphi ({\mathbb R}^n)= M^{p,q}_\psi({\mathbb
R}^n)$ for all $\varphi,\psi \in {\mathcal S}({\mathbb  R}^n) \setminus
\{ 0 \}$ (\cite[Theorem 6.1]{Feichtinger}). This property is crucial in the sequel,
since we choose a suitable window function $\varphi$ to estimate the
modulation space norm. If $1 \leq p,q < \infty$ then
${\mathcal S}({\mathbb R}^n)$ is dense in $M^{p,q}$ (\cite[Theorem
6.1]{Feichtinger}). We also note $L^2 =M^{2,2}$, and $M^{p_1,q_1}
\hookrightarrow M^{p_2,q_2}$ if $p_1 \leq p_2, q_1 \leq q_2$
(\cite[Proposition 6.5]{Feichtinger}). Let us define by ${\mathcal
M}^{p,q}({\mathbb R}^n)$ the completion of ${\mathcal S}({\mathbb R}^n)$
under the norm $\| \cdot \|_{M^{p,q}}$. Then ${\mathcal
M}^{p,q}({\mathbb R}^n) = M^{p,q}({\mathbb R}^n)$ for $1 \leq p,q <
\infty$. Moreover, the complex interpolation theory for these spaces
reads as follows: Let $0< \theta <1$ and $1 \leq  p_i , q_i  \leq
\infty, i = 1, 2$. Set $1/p=(1- \theta)/p_1+ \theta/p_2$, $1/q = (1-
\theta)/q_1 + \theta /q_2$, then $({\mathcal M}^{p_1,q_1}, {\mathcal
M}^{p_2,q_2})_{[\theta]}={\mathcal M}^{p,q}$
(\cite[Theorem 6.1]{Feichtinger}, \cite[Theorem 2.3]{Wang Huang}).
We refer to \cite{Feichtinger} and \cite{Grochenig 2001} for more details.

\section{Key Lemmas}
The orbit of the classical mechanics corresponding to \eqref{SE}
 is described by the system of ordinary differential equations
    \begin{align}
    \label{ODE}
    \begin{cases}
    \dfrac{d}{ds}f(s)=g (s), \\
    \dfrac{d}{ds}g (s)=-(\nabla _x V)(s,f(s)),
    \end{cases}
    \end{align}
 where $f:\re\rightarrow \re^n$ and $g:\re\rightarrow \re^n$ (see also 
 Fujiwara \cite{Fujiwara}).
So, we prepare two lemmas which give some properties
 of the solutions to the system of ordinary differential equations.
Following lemma is used in the proof of Theorem \ref{p-p estimate}.

\begin{lemma}
\label{hensuuhenkan}
Let $V\in C^{\infty} (\re\times \re^{n})$ satisfy \eqref{assumption_V}
 for all multi-indices $\alpha$ with $|\alpha|\ge 2$.
Suppose that $f(s;t,x,\xi)$ and $g(s;t,x,\xi)$ are solutions to \eqref{ODE}
 satisfying $f(t)=x$ and $g(t)=\xi$ and put 
    \begin{align*}
    &M(s;t,x,\xi)=(w_{i,j})\notag\\
    &=\!\left(
    \begin{array}{cccccc}
    \!\dfrac{\partial f_1(s;t,x,\xi)}{\partial x_1}\! &\!\!\cdots\!\!
    &\!\dfrac{\partial f_1(s;t,x,\xi)}{\partial x_n}
    &\!\dfrac{\partial f_1(s;t,x,\xi)}{\partial \xi_1} \! &\!\!\cdots\!\!   
    &\!\dfrac{\partial f_1(s;t,x,\xi)}{\partial \xi_n}\!  \\
    \vdots&&\vdots&\vdots&&\vdots\\
    \!\dfrac{\partial f_n(s;t,x,\xi)}{\partial x_1} \! &\!\!\cdots\!\!
    &\!\dfrac{\partial f_n(s;t,x,\xi)}{\partial x_n}
    &\!\dfrac{\partial f_n(s;t,x,\xi)}{\partial \xi_1} \!&\!\!\cdots\!\!
    &\!\dfrac{\partial f_n(s;t,x,\xi)}{\partial \xi_n}\!  \\
    \!\dfrac{\partial g_1(s;t,x,\xi)}{\partial x_1} \! &\!\!\cdots\!\! 
    &\!\dfrac{\partial g_1(s;t,x,\xi)}{\partial x_n}
    &\!\dfrac{\partial g_1(s;t,x,\xi)}{\partial \xi_1} \! &\!\!\cdots\!\! 
    &\!\dfrac{\partial g_1(s;t,x,\xi)}{\partial \xi_n}\!  \\
    \vdots&&\vdots&\vdots&&\vdots\\
    \!\dfrac{\partial g_n(s;t,x,\xi)}{\partial x_1} \! &\!\!\cdots\!\!
    &\!\dfrac{\partial g_n(s;t,x,\xi)}{\partial x_n}
    &\!\dfrac{\partial g_n(s;t,x,\xi)}{\partial \xi_1} \! &\!\!\cdots\!\!
    &\!\dfrac{\partial g_n(s;t,x,\xi)}{\partial \xi_n}\!  \\
    \end{array}
    \right) .
    \end{align*}
Then ${\rm det}M(s;t,x,\xi)=1$ for all $s,t,x$ and $\xi$.
\end{lemma}

It is easy to prove this lemma by the standard method, but we give the
 proof for reader's convenience in Appendix.
Next lemma plays an important role in the proof of Theorem 
\ref{p-q estimate}.

\begin{lemma}
\label{bra-estimate}
Let $f(s;t,x,\xi)$ and $g(s;t,x,\xi)$ be solutions to \eqref{ODE} with
 $f(t)=x$ and $g(t)=\xi$. If $V\in C^\infty (\re\times \re^{n})$ satisfies 
 \eqref{assumption_V} for all multi-indices $\alpha$ with $|\alpha|\ge 1$,
 then there exist $C_1, C_2>0$ such that 
    \begin{align}
    \label{y-f(s)}
    \dfrac{1}{\bra{y-f(s;t,x,\xi)}}\le \dfrac{C_1(1+|t-s|^2)}{
    \bra{y-x+(t-s)\xi}}
    \end{align}
 and 
    \begin{align}
    \label{eta-g(s)} 
    \dfrac{1}{\bra{\eta -g(s;t,x,\xi)}}
    \le \dfrac{C_2(1+|t-s|)}{\bra{\eta -\xi}}.
    \end{align}
\end{lemma}

\begin{proof} First, we show \eqref{y-f(s)}.
Since $f(s;t,x,\xi)$ and $g(s;t,x,\xi)$ solve \eqref{ODE}, we have 
    \begin{align}
    \label{f(s)}
     f(s;t,x,\xi)
    &=f(t;t,x,\xi)+\int_t^s g(\tau ;t,x,\xi)d\tau \notag\\
    &=x + \int_t^s \left( g(t;t,x,\xi) - \int_t^{\tau}  
     (\nabla_xV)(\sigma,f(\sigma;t,x,\xi))d\sigma\right)d\tau \notag\\
    &=x+(s-t)\xi -\int_s^t\int_s^{\sigma}(\nabla_x V)(\sigma,
    f(\sigma;t,x,\xi))d\tau d\sigma \notag\\
    &=x+(s-t)\xi -\int_s^t (\sigma -s)(\nabla_x V)(\sigma,
    f(\sigma;t,x,\xi))d\sigma .
    \end{align}
Since $V$ satisfies \eqref{assumption_V} for all multi-indices $\alpha$
 with $|\alpha|\ge 1$, we have
    \begin{align}
    \label{V_bibun}
    |(\partial_{x_j}V)(\sigma ,
    f(\sigma;t,x,\xi))|\le 2C
    \end{align}  
 for $j=1,2,\ldots ,n$. 
By \eqref{f(s)} and \eqref{V_bibun}, we have 
    \begin{align*}
     |y-x+(t-s)\xi |
    &\le |y-f(s;t,x,\xi)|+|f(s;t,x,\xi)-x+ (t-s)\xi| \notag\\
    &\le |y-f(s;t,x,\xi)|+\left| \int_s^t (\sigma -s)(\nabla_x V)(\sigma,
     f(\sigma;t,x,\xi))d\sigma \right|\notag\\
    &\le |y-f(s;t,x,\xi)|+2\sqrt{n}C\left|\int_s^t |\sigma -s|
      d\sigma\right| \notag\\
    &= |y-f(s;t,x,\xi)|+\sqrt{n}C|t-s|^2 .
    \end{align*}
So, we have
    \begin{align*}
     \bra{y-x+(t-s)\xi}
     &\le \{ 1+2(|y-f(s;t,x,\xi)|^2+nC^2|t-s|^4)\}^{1/2}\\
     &\le \sqrt{2}(1+\sqrt{n} C|t-s|^2) \bra{y-f(s;t,x,\xi)}.
    \end{align*}
Putting $C_1=\sqrt{2}\,{\rm max}\{ 1, \sqrt{n} C\}$, 
 we obtain \eqref{y-f(s)}.\\

Next, we show \eqref{eta-g(s)}. 
Since $f(s;t,x,\xi)$ and $g(s;t,x,\xi)$ solve \eqref{ODE}, we have 
    \begin{align*}
     g (s;t,x,\xi)
    &=g (t;t,x,\xi)-\int_t^s \nabla_x V(\sigma,
     f(\sigma;t,x,\xi))d\sigma  \notag\\
    &=\xi + \int_s^t \nabla_x V(\sigma, f(\sigma;t,x,\xi))d\sigma
    \end{align*} 
 and thus, 
    \begin{align*}
    |\eta -\xi|
    &\le |\eta -g(s;t,x,\xi)|+|g(s;t,x,\xi)-\xi|\\
    &\le |\eta -g(s;t,x,\xi)|+\left| \int_s^t (\nabla_xV)(\sigma,
    f(\sigma;t,x,\xi )) d\sigma \right|\\
    &\le |\eta -g (s;t,x,\xi)|+2\sqrt{n}C|t-s| .
    \end{align*}
Hence, we have
    \begin{align*}
    \bra{\eta-\xi}
    &\le \left\{ 1+2(|\eta -g(s;t,x,\xi)|^2+4nC^2|t-s|^2)
    \right\}^{\frac{1}{2}}\\
    &\le \sqrt{2}(1+2\sqrt{n}C|t-s|)\bra{\eta -g(s;t,x,\xi)}.
    \end{align*}
Putting $C_2=\sqrt{2}\, {\rm max}\{1,2\sqrt{n}C\}$,
 we obtain \eqref{eta-g(s)}.
\end{proof}

\section{Proof of Theorem \ref{p-p estimate}}
We only consider the case $t\in [0,T]$. We can treat the case 
 $t\in[-T,0]$ in the same way.
First, by using wave packet transform, we transform \eqref{SE} into a
 first order partial differential equation and a lower order term.
By integration by parts, we have
    \begin{align}
    \label{t_part}
    &W_{\varphi (t,\cdot)}(i\partial_t u)(t,x,\xi)=i\partial_t
    W_{\varphi (t,\cdot)}u(t,x,\xi)+
    W_{i\partial_t\varphi (t,\cdot)}u(t,x,\xi)
    \end{align}
and 
    \begin{align}
    \label{delta_part}
    &W_{\varphi(t,\cdot)}\Big(\frac{1}{2}\Delta u\Big)(t,x,\xi)\notag\\
    &=W_{\frac{1}{2}\Delta \varphi(t,\cdot)}u(t,x,\xi)+
    i\xi\cdot\nabla_xW_{\varphi (t,\cdot)}u(t,x,\xi)-\frac{|\xi|^2}{2}
    W_{\varphi (t,\cdot)}u(t,x,\xi),
    \end{align}
 where $\varphi (t,x)=e^{\frac{1}{2}it\Delta}\varphi_0 (x)$.
Applying Taylor's theorem to $V (t,\cdot)$, we have, by integration by
 parts,
    \begin{align} 
    \label{V_part}
    &W_{\varphi(t,\cdot)}(Vu)(t,x,\xi)\notag\\
    &=\int_{\re^n} \overline{\varphi (t,y-x)}\Big( V(t,x)+\nabla_xV(t,x)\cdot
     (y-x)\notag\\
    &\quad \quad \quad +\sum_{j,k=1}^n (y_j-x_j)(y_k-x_k)V_{jk}(t,x,y)\Big)
     u(t,y)e^{-iy\cdot\xi}dy\notag\\
    &=\{V(t,x)+i\nabla_x V(t,x)\cdot \nabla_\xi-\nabla_x V(t,x)\cdot x\}
     W_{\varphi(t,\cdot)}u(t,x,\xi)+Ru(t,x,\xi),
    \end{align}
 where
    \begin{multline}
    \label{Ru}
    Ru(t,x,\xi)\\  =  \sum_{j,k=1}^n \int_{\re^n} \overline{\varphi
    (t,y-x)}V_{jk}(t,x,y)(y_j-x_j)(y_k-x_k) u(t,y)e^{-iy\cdot\xi}dy
    \end{multline}
 and 
    \begin{align}
    \label{V_jk}
     V_{jk}(t,x,y)=\int_0^1 \partial_{x_j}\partial_{x_k}
    V(t,x+\theta(y-x))(1-\theta)d\theta .
    \end{align} 
Since 
 $i\partial_t\varphi (t,x) +\frac{1}{2}\Delta \varphi(t,x) =0$,
 we have
    \begin{align}
    \label{phi_window}
    W_{i\partial_t\varphi (t,\cdot)}u(t,x,\xi)+
    W_{\frac{1}{2}\Delta \varphi(t,\cdot)}u(t,x,\xi)=0.
    \end{align}
Combining \eqref{t_part}, \eqref{delta_part}, \eqref{V_part} and 
 \eqref{phi_window}, the initial value problem \eqref{SE} is transformed to 
    \begin{equation*}
    \begin{cases}
    \Big( i\partial_t +i\xi\cdot\nabla_x -i\nabla_x V(t,x)\cdot \nabla_\xi
- \frac{1}{2}|\xi|^2
    -V(t,x)  \vspace{1mm}\\
    \hspace{20mm} +\nabla_x V(t,x)\cdot x\Big) 
    W_{\varphi (t,\cdot)}u(t,x,\xi)
    -Ru(t,x,\xi)=0,\vspace{1mm}\\ 
    W_{\varphi (0,\cdot)}u(0,x,\xi)=W_{\varphi_0}u_0(x,\xi).
    \end{cases}
    \end{equation*}
By the method of characteristics, we obtain
    \begin{multline}
    \label{kainohyouji}
    W_{\varphi (t,\cdot)}u(t,x,\xi)
    =e^{-i\int_0^t h(s;t,x,\xi)ds}
    \bigg( W_{\varphi_0} u_0(f(0;t,x,\xi),g(0;t,x,\xi))\\
    -i\int_0^t e^{i\int_0^\tau h(s;t,x,\xi)ds}Ru(\tau ,f(\tau;t,x,\xi), g
     (\tau;t,x,\xi)) d\tau\bigg),
    \end{multline}
 where $f(s;t,x,\xi)$ and $g(s;t,x,\xi)$ are solutions to \eqref{ODE}
 with $f(t)=x$ and $g(t)=\xi$, and
\begin{multline*}
h(s;t,x,\xi)\\
=\frac{1}{2}|g (s;t,x,\xi)|^2+
    V(s,f(s;t,x,\xi))-\nabla_x V(s,f(s;t,x,\xi))\cdot f(s;t,x,\xi).
\end{multline*}
By taking $L^p$-norm with respect to $x$ and $\xi$ on both sides of 
 \eqref{kainohyouji}, we have
    \begin{align}
    \label{solution}
    \|u(t,\cdot)\|_{M^{p,p}_{\varphi (t,\cdot)}}
    =\|W_{\varphi (t,\cdot)}u(t,x,\xi)\|_{L^p_{x,\xi}}
   \le \|I_1\|_{L^p_{x,\xi}}+\int_0^t \|I_2\|_{L^p_{x,\xi}}  d\tau ,
    \end{align}
 where 
    \begin{equation}
     \label{I}
     I_1=W_{\varphi_0}u_0(f(0;t,x,\xi),g(0;t,x,\xi)), 
     ~~ I_2=Ru(\tau, f(\tau;t,x,\xi),g(\tau;t,x,\xi)).
    \end{equation}

Now we consider the change of variables
 $X=f(0;t,x,\xi)$ and $\Xi =g(0;t,x,\xi)$. 
From Lemma \ref{hensuuhenkan} and the implicit function theorem, 
 we have
 $$\left| \dfrac{\partial (x,\xi)}{\partial (X,\Xi)}
  \right|=1.$$
So it follows that 
    \begin{align}
    \label{I1}
    \| I_1 \|_{L^p_{x,\xi}}
    =\left(\iint_{\re^{2n}} |W_{\varphi_0} u_0(X,\Xi)|^p 
    \left| \dfrac{\partial (x,\xi)}{\partial (X,\Xi)}\right|
    dXd\Xi\right)^{\frac{1}{p}}=\| u_0 
    \|_{M^{p,p}_{\varphi_0}}.
    \end{align}

On the other hand, from \eqref{Ru} and 
the inversion formula of wave packet transform for $u$,
 we have
    \begin{multline*}
    Ru(t,x,\xi)
    =
    \frac{1}{\|\varphi(t,\cdot)\|^2_{L^2}}
    \sum_{j,k=1}^n \iiint_{\re^{3n}} \varphi_{jk}(t,y-x) V_{jk}
    (t,x,y) \varphi(t,y-z) \\
    \times W_{\varphi(t,\cdot)}u(t,z,\eta) 
    e^{iy\cdot(\eta -\xi)} dz\dbar\eta dy,
    \end{multline*}
 where $\varphi_{jk}(t,y)=y_j y_k \overline{\varphi (t,y)}$. 
Take $N\in \na$ satisfying $2N>n$.
From
$$(1-\Delta_y)^N e^{iy\cdot (\eta -g(\tau;t,x,\xi))}=\bra{\eta
 -g(\tau;t,x,\xi)}^{2N}e^{iy\cdot (\eta -g (\tau;t,x,\xi))},$$
 we have
    \begin{align*}
   & \|I_2\|_{L^p_{x,\xi}}\\
   &=\| Ru(\tau,f(\tau;t,x,\xi),g(\tau;t,x,\xi))\|_{L^p_{x,\xi}}\\
    &\le \dfrac{1}{\|\varphi (\tau,\cdot)\|^2_{L^2}}
    \sum_{j,k=1}^n \bigg\| \iiint_{\re^{3n}} \Big| (1-\Delta_y)^N
    \Big\{\varphi_{jk}(\tau,y-f(\tau;t,x,\xi))  \\
    &~\times V_{jk}(\tau,f(\tau;t,x,\xi),y) \varphi (\tau,y-z) \Big\}\Big|
    \dfrac{|W_{\varphi (\tau,\cdot)} u(\tau,z,\eta)|}{
    \bra{\eta -g(\tau;t,x,\xi)}^{2N}} 
    dz\dbar\eta dy \bigg\|_{L^p_{x,\xi}} \\
    &\le \dfrac{1}{\|\varphi (\tau,\cdot)\|^2_{L^2}}
    \sum_{j,k=1}^n\sum_{|\beta_1|+|\beta_2|+|\beta_3|\le 2N}
     \bigg\| \iiint_{\re^{3n}}
    \Big| \partial^{\beta_1}_y\varphi_{jk}(\tau,y-f(\tau;t,x,\xi))       \\
    &~\times \partial^{\beta_2}_y V_{jk}(\tau,f(\tau;t,x,\xi),y)
    \partial^{\beta_3}_y\varphi (\tau,y-z) \Big|
    \dfrac{|W_{\varphi (\tau,\cdot)} u(\tau,z,\eta)|}{
    \bra{\eta -g(\tau;t,x,\xi)}^{2N}} 
    dz\dbar\eta dy \bigg\|_{L^p_{x,\xi}}.
   \end{align*}
Since $|\partial^{\beta_2}_y V_{jk}(\tau, f(\tau;t,x,\xi),y)|\le C_{\beta_2}$
 for $C_{\beta_2}>0$, we have, by the change of variables  
 $X=f(\tau;t,x,\xi)$ and $\Xi =g(\tau;t,x,\xi)$, Young's inequality
 and Lemma \ref{hensuuhenkan},
    \begin{align}
    \label{I2}
    &\|I_2\|_{L^p_{x,\xi}}\notag\\
    &\le \dfrac{1}{\|\varphi (\tau,\cdot)\|^2_{L^2}}
    \sum_{j,k=1}^n\sum_{|{\beta_1}|+|{\beta_2}|+|{\beta_3}|\le 2N} 
    C_{\beta_2}
    \bigg\{ \iint_{\re^{2n}}  \bigg( \iint_{\re^{2n}}
     \dfrac{|\partial^{\beta_1}_y \varphi_{jk}(\tau,y-X)|}{\bra{\eta
    -\Xi}^{2N}}\notag\\
    &\quad\times \int_{\re^n} |\partial^{\beta_3}_y\varphi (\tau,y-z)
    W_{\varphi (\tau,\cdot)} 
    u(\tau,z,\eta)| dz
    \dbar\eta dy \bigg)^p \left| \dfrac{\partial (x,\xi)}{\partial
    (X,\Xi)}\right| dXd\Xi
    \bigg\}^{\frac{1}{p}}\notag\\
    &\le \dfrac{1}{\|\varphi (\tau,\cdot)\|^2_{L^2}}
    \sum_{j,k=1}^n\sum_{|{\beta_1}|+|{\beta_2}|+|{\beta_3}|\le
     2N}\!\!\!\! C_{\beta_2}
    \left\| \dfrac{\partial^{\beta_1}_y
    \varphi_{jk}(\tau,y)}{\bra{\eta}^{2N}}
    \right\|_{L^1_{y,\eta}}\!\!\!\!
    \| \partial^{\beta_3}_y \varphi (\tau,\cdot)\|_{L^1}\notag\\
    &\quad \times\|W_{\varphi(\tau,\cdot )} u(\tau,z,\eta)
    \|_{L^p_{z,\eta}}\notag\\
    &\le C_{T}    \|u(\tau,\cdot )\|_{M^{p,p}_{\varphi (\tau,\cdot)}}
    \end{align}
 for $t\in [0,T]$ and $\tau\in [0,t]$.
From \eqref{solution}, \eqref{I1} and \eqref{I2}, we have 
    \begin{align*}
    \|u(t,\cdot)\|_{M^{p,p}_{\varphi (t,\cdot)}}
    \le \|u_0\|_{M^{p,p}_{\varphi_0}}+C_{T}
     \int_0^t \| 
    u(\tau,\cdot)\|_{M^{p,p}_{\varphi (\tau,\cdot)}}d\tau
    \end{align*} 
 for $t\in [0,T]$.
Then Gronwall's inequality yields 
$\| u(t,\cdot)\|_{M^{p,p}_{\varphi (t,\cdot)}}
\le C_T \| u_0\|_{M^{p,p}_{\varphi_0}}$ for $t\in [0,T]$.
\hfill $\square$

\section{Proof of Theorem \ref{p-q estimate}}
We only consider the case $t\in [0,T]$, since we can treat the case
 $t\in [-T,0]$ in the same way. 
First, we consider the case $(p,q)=(\infty, 1)$, next 
 $(p,q)=(1,\infty)$ and finally general $(p,q)$.

In the proof of Theorem \ref{p-p estimate}, we have already obtained
 $$|W_{\varphi (t,\cdot)}u(t,x,\xi)|\le |I_1|+\int_0^t|I_2|\,d\tau,$$
 where $I_1$ and $I_2$ are defined by \eqref{I}.
Take $N\in \na$ satisfying $2N>n$.
From 
    $$(1-\Delta_y)^N e^{iy\cdot (\eta-g(0;t,x,\xi))}=\bra{\eta
    -g(0;t,x,\xi)}^{2N}e^{iy\cdot (\eta-g(0;t,x,\xi))},$$
we have, by the inversion formula of wave packet transform for $u_0$ 
 and \eqref{eta-g(s)} in Lemma \ref{bra-estimate},
    \begin{align}
    \label{I1_est}
     |I_1|
    &=\dfrac{1}{\|\varphi_0\|_{L^2}^2}
      \bigg| \iiint_{\re^3} \overline{\varphi_0(y-f(0;t,x,\xi))} \varphi_0
     (y-z)\notag\\
     &\hspace{40mm}
      \times W_{\varphi_0}u_0(z,\eta) e^{iy\cdot(\eta -g(0;t,x,\xi))}
      dz\dbar\eta dy\bigg| \notag\\
    &\le \dfrac{1}{\|\varphi_0\|_{L^2}^2} 
     \iiint_{\re^{3n}} |(1-\Delta_y)^N \{\overline{\varphi_0(y-f(0;t,x,\xi))}
     \varphi_0 (y-z)\}|\notag\\
    &\hspace{40mm}\times \dfrac{|W_{\varphi_0}u_0(z,\eta)|}{
     \bra{\eta
     -g(0;t,x,\xi)}^{2N}} dz\dbar\eta dy \notag\\
    &\le C (1+|t|)^{2N}\!\! \sum_{|{\beta_1}|+|\beta_2|\le 2N}
     \iiint_{\re^{3n}} \big|\overline{\partial^{\beta_1}_y 
    \varphi_0(y-f(0;t,x,\xi))}\big|\notag \\
         &\hspace{40mm} \times |\partial^{\beta_2}_y\varphi_0 (y-z)|
|\dfrac{|W_{\varphi_0}u_0(z,\eta)|}{
     \bra{\eta -\xi}^{2N}} dz\dbar\eta dy.
    \end{align}
By Fubini's theorem, we have
    \begin{align*}
    &\|\|I_1\|_{L^\infty_{x}}\|_{L^1_\xi}\\
    &\le C (1+|t|)^{2N}
     \sum_{|{\beta_1}|+|{\beta_2}|\le 2N}\|\partial^{\beta_2}_y
    \varphi_0\|_{L^1}\\
    &\qquad\times \bigg\|\bigg\| \iint_{\re^{2n}} |\partial^{\beta_1}_y
     \overline{\varphi_0 (y-f(0;t,x,\xi)} |\dfrac{\|
     W_{\varphi_0}u_0(z,\eta)\|_{L^\infty_z}}{ 
    \bra{\eta -\xi}^{2N}} \dbar\eta dy 
    \bigg\|_{L^\infty_x}\bigg\|_{L^1_\xi}\notag\\
   &\le C (1+|t|)^{2N} 
    \sum_{|{\beta_1}|+|{\beta_2}|\le 2N}\|\partial^{\beta_2}_y
    \varphi_0\|_{L^1}
    \|\partial^{\beta_1}_y\varphi_0\|_{L^1}
    \|\|W_{\varphi_0}u_0(z,\eta)\|_{L^\infty_z}\|_{L^1_\eta}\\
   &\le C_T\|u_0\|_{M^{\infty,1}_{\varphi_0}}
    \end{align*}
for $t\in [0,T]$.
Similarly, we have
    \begin{align*}
    |I_2|
   & \le \dfrac{1}{\|\varphi (\tau,\cdot)\|^2_{L^2}} 
    \sum_{j,k=1}^n   \iiint_{\re^{3n}}  \Big| (1-\Delta_y)^N  
    \{ \varphi_{jk}(\tau,y-f(\tau;t,x,\xi)) \notag\\
   &\quad\quad \times V_{jk}(\tau,f(\tau;t,x,\xi),y)
    \varphi (\tau,y-z)\} \Big|
    \dfrac{|W_{\varphi (\tau,\cdot )}u(\tau,z,\eta)|}{
     \bra{\eta -g(\tau ;t,x,\xi)}^{2N}}  dz \dbar\eta  dy\notag\\
   & \le \dfrac{C(1+|t-\tau |)^{2N}}{\|\varphi (\tau,\cdot)\|^2_{L^2}}\\
   &\qquad \times \sum_{j,k=1}^n  \sum_{|{\beta_1}|+|\beta_2|+|{\beta_3}|\le 2N} 
    \iiint_{\re^{3n}}
    |\partial^{\beta_1}_y\varphi_{jk}(\tau,y-f(\tau;t,x,\xi))| \notag\\
   &\quad\quad \times |\partial^{\beta_2}_yV_{jk}(\tau,f(\tau;t,x,\xi),y) 
    \partial^{\beta_3}_y\varphi (\tau,y-z)| 
    \dfrac{|W_{\varphi (\tau,\cdot )}u(\tau,z,\eta)|}{
     \bra{\eta -\xi}^{2N}}  dz \dbar\eta  dy,
    \end{align*}
 where $\varphi_{jk}(t,y)=y_j y_k \overline{\varphi(t,y)}$ and 
 $V_{jk}$ is defined by \eqref{V_jk}.
Since 
 $$|\partial^{\beta_2}_y V_{jk}(\tau,f(\tau;t,x,\xi),y)| \le C_{\beta_2}$$
 for $C_{\beta_2} >0$,
 we have
    \begin{align*}
    \|\|I_2\|_{L^{\infty}_{x}}\|_{L^1_\xi}
    & \le \dfrac{C(1+|t-\tau|)^{2N}}{\|\varphi (\tau,\cdot)\|^2_{L^2}}\\
    & ~~\times\sum_{j,k=1}^n \sum_{|{\beta_1}|+|\beta_2|+|{\beta_3}|\le 2N} 
    \!\!\!\! C_{\beta_2}
    \bigg\|\bigg\| \iiint_{\re^{3n}}  
    |\partial^{\beta_1}_y\varphi_{jk}(\tau,y-f(\tau;t,x,\xi))|\\
   & ~~\times  |\partial^{\beta_3}_z \varphi (\tau,y-z) | 
    \dfrac{|W_{\varphi (\tau,\cdot )}u(\tau,z,\eta)|}{
     \bra{\eta -\xi}^{2N}}  dz \dbar\eta  dy
    \bigg\|_{L^\infty_x}\bigg\|_{L^1_\xi}\\
    &\le \dfrac{C'(1+T)^{2N}}{\|\varphi (\tau,\cdot)\|^2_{L^2}}
    \sum_{j,k=1}^n \sum_{|{\beta_1}|+|\beta_2|+|{\beta_3}|\le 2N} 
    C_{\beta_2}  \|\partial^{\beta_3}_z \varphi (\tau,z)\|_{L^1_z}\\
    & ~~
    \times \|\partial^{\beta_1}_y \varphi_{j,k} (\tau,y)\|_{L^1_y}
    \|\bra{\cdot}^{-2N}\|_{L^1}
     \|\|  W_{\varphi (\tau,\cdot)}u(\tau,z,\eta)
    \|_{L^\infty_z}\|_{L^1_\eta}\notag\\
   &\le C'_T \|u(\tau,
 \cdot)\|_{M^{\infty,1}_{\varphi (\tau, \cdot)}}
\end{align*}
for $t\in [0,T]$ and $\tau\in [0,t]$.
Hence, we have 
    \begin{align}
    \label{inf-1_I1_I2} 
     \| u(t,\cdot)\|_{M^{\infty,1}_{\varphi(t,\cdot)}}
    & \le  \|\| I_1\|_{L^\infty_x}\|_{L^1_\xi}+
     \int_0^t \|\| I_2\|_{L^\infty_x}\|_{L^1_\xi}d\tau\notag\\
    &\le C_T\|u_0\|_{M^{\infty,1}_{\varphi_0}}+
     C'_T\int_0^t \| u(\tau
     ,\cdot)\|_{M^{\infty,1}_{\varphi(\tau ,\cdot)}}d\tau
    \end{align}
 for $t\in [0,T]$.
Applying Gronwall's inequality to \eqref{inf-1_I1_I2}, we obtain 
    \begin{align}
    \label{inf-1}
    \| u(t,\cdot) \|_{M^{\infty,1}_{\varphi (t,\cdot)}}
    \le C''_{T} \| u_0 \|_{M^{\infty,1}_{\varphi_0}}
    \end{align}
for $t\in [0,T]$.

Next, we consider $(p,q)=(1,\infty)$.
Take $N\in\na$ satisfying $2N>n$. For all multi-indices ${\beta_1}$, 
we have
    \begin{align}
    \label{estimate_psi_f}
    &\quad \|\partial^{{\beta_1}}_{y} \varphi 
     (\tau,y-f(\tau;t,x,\xi))\|_{L^1_x}\notag\\
    &\le  C (1+|t-\tau|^{2})^{2N} \int_{\re^n} 
    \dfrac{\bra{y-f(\tau;t,x,\xi)}^{2N}}{\bra{y-x+(t-\tau)\xi}^{2N}}
    |\partial^{{\beta_1}}_{y} \varphi (\tau,y-f(\tau;t,x,\xi))|\, dx\notag\\
    &\le C(1+T^{2})^{2N} \left(\underset{\tau\in [0,T],y\in\re^{n}}{\rm sup}
    \bra{y}^{2N}|\partial^{\beta_1}_y \varphi (\tau,y)|
     \right)\int_{\re^n}\dfrac{1}{\bra{y-x+(t-\tau)\xi}^{2N}}dx\notag\\
    &\le C_T
    \end{align}
 for $t\in [0,T]$ and $\tau\in [0,t]$. 
Here, we have used \eqref{y-f(s)} in Lemma \ref{bra-estimate}.
Thus, we have, by \eqref{I1_est}, \eqref{estimate_psi_f}
 and Fubini's theorem,
    \begin{align*}
    &\|\|I_1\|_{L^1_x}\|_{L^\infty_\xi}\\
    &\le C (1+|t|)^{2N}\sum_{|{\beta_1}|+|\beta_2|\le 2N}\bigg\|
     \iiint_{\re^{3n}}  
    \|\partial^{\beta_1}_y \varphi_0 (y-f(0;t,x,\xi))\|_{L^1_{x}}\\
    &\hspace{45mm}\times
     |\partial^{\beta_2}_y \varphi_0 (y-z) | 
    \dfrac{|W_{\varphi_0}u_0(z,\eta)|}{\bra{\eta
    -\xi}^{2N}}dz\dbar\eta dy
    \bigg\|_{L^\infty_\xi}\\
    &\le C'_T \|u_0\|_{M^{1,\infty}_{\varphi_0}}
\end{align*}
 for $t\in [0,T]$.
In the similar way as above, it follows that 
    \begin{align*}
    &\|\|I_2\|_{L^{1}_{x}}\|_{L^\infty_\xi}\\
    & \le \dfrac{C(1+|t-\tau|)^{2N}}{\|\varphi (\tau,\cdot)\|^2_{L^2}}\\
    &~\times\sum_{j,k=1}^n \sum_{|{\beta_1}|+|\beta_2|+|{\beta_3}|\le 2N} 
    \bigg\|\bigg\| \iiint_{\re^{3n}}  
    |\partial^{\beta_1}_y\varphi_{j,k}(\tau,y-f(\tau;t,x,\xi))|\\
   & ~\times | \partial_y^{\beta_2} V_{j,k}(\tau,f(\tau;t,x,\xi),y)\,
    \partial^{\beta_3}_y \varphi (\tau,y-z) | 
    \dfrac{|W_{\varphi (\tau,\cdot )}u(\tau,z,\eta)|}{
     \bra{\eta -\xi}^{2N}}  dz \dbar\eta  dy
    \bigg\|_{L^1_x}\bigg\|_{L^\infty_\xi}\\
   &\le C''_T \|u(\tau,
 \cdot)\|_{M^{1,\infty}_{\varphi (\tau, \cdot)}}
\end{align*}
for $\tau\in [0,t]$ and $t\in [0,T]$.
Thus, we have
    \begin{align}
    \label{1-inf_I1_I2} 
     \| u(t,\cdot)\|_{M^{1,\infty}_{\varphi(t,\cdot)}}
    & \le  \|\| I_1\|_{L^1_x}\|_{L^\infty_\xi}+
     \int_0^t \|\| I_2\|_{L^1_x}\|_{L^\infty_\xi}d\tau\notag\\
    &\le C'_T\|u_0\|_{M^{1,\infty}_{\varphi_0}}+
     C''_T\int_0^t \| u(\sigma
     ,\cdot)\|_{M^{1,\infty}_{\varphi(\tau,\cdot)}}d\tau
    \end{align}
 for $t\in [0,T]$.
Applying Gronwall's inequality to \eqref{1-inf_I1_I2}, we obtain 
    \begin{align}
    \label{1-inf}
    \| u(t,\cdot) \|_{M^{1,\infty}_{\varphi (t,\cdot)}}
    \le C'''_{T} \| u_0 \|_{M^{1,\infty}_{\varphi_0}}
    \end{align}
for $t\in [0,T]$.

Finally, we consider the general case.
From Theorem \ref{p-p estimate}, we have
    \begin{align}
    \label{p-p}
    \| u(t,\cdot) \|_{M^{p,p}_{\varphi (t,\cdot)}}
    \le C_{T} \| u_0 \|_{M^{p,p}_{\varphi_0}}.
    \end{align}
Combing \eqref{inf-1}, \eqref{1-inf} and \eqref{p-p},
we have, by the complex interpolation theorem for modulation space,
 $$\|u(t,\cdot)\|_{M^{p,q}_{\varphi (t,\cdot)}}\le
 \widetilde{C}_T
\|u_0\|_{M^{p,q}_{\varphi_0}}$$
for $t\in [0,T]$.
Therefore we obtain the desired result.
~\hfill $\square$

\appendix
\section{}
First, We remark that if $V(t,x)\in C^\infty (\re\times\re^n)$ satisfies
 \eqref{assumption_V} for all multi-indices
 with $|\alpha|\ge 2$ then the ordinary differential equation \eqref{ODE}
 with the initial condition $f(t)=x$ and $g(t)=\xi$ has a unique solution
 on $\re$. In fact, the existence of the solution on
 $[t-\sigma,t+\sigma]$ for some $\sigma>0$ is proved by Picard's
 iteration scheme. Here is the outline. Let $f^{(0)}(s)\equiv x$ and
 $g^{(0)}(s)\equiv \xi$ and set 
    \begin{align*}
    f^{(k+1)}(s)\!=\!x+\!\!\int_{t}^s g^{(k)}(\tau)d\tau
    ~~\text{and}~~  g^{(k+1)}(s)
    \!=\!\xi-\!\!\int_{t}^s (\nabla_x V)(\tau, f^{(k)}(\tau))d\tau.
    \end{align*}
Then $\{f^{(k)}\}$ and $\{g^{(k)}\}$ converge uniformly to some
 functions $f(s)$ and $g(s)$ on $[t-\sigma ,t+\sigma]$ and the functions
 $f(s)$ and $g(s)$ satisfy the initial value problem and belong to 
 $C^\infty([t-\sigma ,t+\sigma])$. 
Moreover, by using following Lemma A.1, we can show that 
 above fact holds not only on $[t-\sigma ,t+\sigma]$ but also on
 $\re$, easily.
Lemma A.1 is also used in the proof of Lemmas A.2 and A.3.

\begin{lemmaA1}
Let $V\in C^\infty (\re\times \re^{n})$ satisfy
 \eqref{assumption_V} for all multi-indices with $|\alpha|\ge 2$.
Then, for all multi-indices $\beta$ with $|\beta|\ge 1$,
 $\partial^\beta_x V(t,x)$ is Lipschitz continuous with respect to
 $x$, more precisely, there exists $C_\beta >0$ such that 
    \begin{align*}
    \big| (\partial^\beta_x V)  \big(t,y\big)
    -(\partial^\beta_x V) \big( t,z\big)  \big| 
    \le  C_\beta n\, ||y-z||_{\infty}
    \end{align*}
 for all $t\in\re$, $y,z\in \re^n$.
\end{lemmaA1}

\begin{proof} 
Let $t\in\re$, $y=(y_1,\ldots, y_n)\in \re^n$ and 
 $z=(z_1,\ldots, z_n)\in \re^n$.
Since $|y_k-z_k|\le || y-z||_{\infty}$ for $k=1,\ldots ,n$,
 it is enough to show that 
   \begin{align*}
   \big| (\partial^\beta_x V)\big(t,y\big)-(\partial^\beta_x
    V)\big( t,z\big)\big|
   \le C_\beta \sum^{n}_{k=1}\big| y_k-z_k \big| .
   \end{align*}
Set $F(\theta)=(\partial^\beta _xV)(t,z+\theta (y-z))$.
We note that $F(\theta)\in C^{\infty}([0,1])$,
 $F(0)=(\partial^\beta_xV)(t,z)$ and $F(1)=(\partial^\beta_xV)(t,y)$.
By the fundamental theorem of calculus, we have
    \begin{align}
    \label{vbibunnosa}
    (\partial^\beta_xV)(t,y)-(\partial^\beta_xV)(t,z)
    & = \int_0^1\frac{d}{d\theta}F(\theta) d\theta \notag \\
    & = \int_0^1 \sum_{k=1}^n (y_k-z_k)
    (\partial_{x_k}\partial^\beta_x V) (t,z+\theta (y-z))d\theta .
    \end{align}
Since $V$ satisfies \eqref{assumption_V} for $|\alpha |\ge 2$, we obtain
    \begin{align}
    \label{vnoestimate-2}
    \big| (\partial_{x_k}\partial^\beta_x V) \big( t,z+
    \theta (y-z)\big)\big|  \le C_\beta 
    \end{align}
 for $k=1,2,\ldots ,n$. 
Combining \eqref{vbibunnosa} and
 \eqref{vnoestimate-2}, we obtain the desired result.
 \end{proof}

Next, we establish one more lemma relating to the Lemma A.3.

\begin{lemmaA2}
Let $h\in\re\backslash\{ 0\}$, $T>0$ and  $V\in C^\infty (\re\times \re^{n})$
 satisfy \eqref{assumption_V} for all multi-indices $\alpha$ with
 $|\alpha|\ge 2$. Suppose that $f(s;t,x,\xi)$ and $g(s;t,x,\xi)$ are
 solutions to \eqref{ODE} with $f(t)=x$ and $g(t)=\xi$. For 
 $k=1,\ldots,n$, we set  
    \begin{multline*}
    \phi_{h,k} (s;t,x,\xi)\\ =\big(f(s;t,x+he_k,\xi)-f(s;t,x,\xi),\, 
    g(s;t,x+he_k,\xi)-g(s;t,x,\xi)\big)
    \end{multline*}
 and 
    \begin{multline*}
    \psi_{h,k} (s;t,x,\xi)\\ =\big(f(s;t,x,\xi+he_k)-f(s;t,x,\xi),\, 
    g(s;t,x,\xi+he_k)-g(s;t,x,\xi)\big),
    \end{multline*}
 where  $e_k=(0,\ldots ,0,\overset{\underset{\smallsmile}{k}}{1},0,
 \ldots,0)\in \re^{n}$.
Then 
    \begin{align*}
    \lim_{h\to 0}\underset{|s-t|\le T}{\rm sup} 
    \|\phi_{h,k} (s;t,x,\xi) \|_\infty \!=0
    \quad and\quad 
    \lim_{h\to 0}\underset{|s-t|\le T}{\rm sup}
    \|\psi_{h,k} (s;t,x,\xi) \|_\infty \!=0.
    \end{align*}
\end{lemmaA2}

\begin{proof}
Here, we only show that 
 $\lim_{h\to 0}{\rm sup}_{|s-t|\le T}\|\phi_{h,1} (s;t,x,\xi)\|_\infty =0$.
We can treat the other cases in the same way.
Put 
    \begin{align*}
    F(s;t,x,\xi)&=f(s;t,x+he_1,\xi)-f(s;t,x,\xi)
    \end{align*}
 and 
    \begin{align*}
    G(s;t,x,\xi)&=g(s;t,x+he_1,\xi)-g(s;t,x,\xi).
    \end{align*}
Since $F(t;t,x,\xi)=f(t;t,x+he_1,\xi)-f(t;t,x,\xi)=he_1$ and
 $$\frac{d}{ds}F(s;t,x,\xi)=g
 (s;t,x+he_1,\xi)-g(s;t,x,\xi)=G(s;t,x,\xi),$$
 we have 
 $$F(s;t,x,\xi)=he_1+\int_t^s G (\tau;t,x,\xi)\,d\tau .$$
Thus, we have
    \begin{align}
    \label{x-tilde}
    \| F(s;t,x,\xi) \|_\infty
    \le |h|+ \left|\int_t^s \| G(\tau ;t,x,\xi)\|_\infty d\tau
    \right| .
    \end{align}

On the other hand, since $G(t;t,x,\xi)=0$ and
    \begin{align*}
    \dfrac{d}{ds}G(s;t,x,\xi) =-\nabla_x 
    V\big( s,f(s;t,x+he_1,\xi)\big) +\nabla_xV\big( s,f(s;t,x,\xi)\big),
    \end{align*}
 we have 
    $$G (s;t,x,\xi)
    \!=\! - \int_t^s \Big\{ \nabla_x V\big(\tau,f(\tau ;t,x+he_1,\xi)\big)-
    \nabla_xV\big(\tau ,f(\tau ;t,x,\xi)\big)\Big\} d\tau. $$ 
By Lemma A.1, there exists $C>0$ such that 
    \begin{align}
    \label{xi-tilde}
    &\| G(s;t,x,\xi) \|_\infty\notag\\
    &\le \left|\int_t^s \| \nabla_x V( \tau,f(\tau
    ;t,x+he_1,\xi))- \nabla_xV(\tau ,f(\tau ;t,x,\xi))
    \|_\infty d\tau \right| \notag\\
    &\le Cn \left|\int_t^s \|F(\tau ;t,x,\xi)\|_\infty d\tau \right|.
    \end{align}
From \eqref{x-tilde} and \eqref{xi-tilde}, we have
    \begin{align}
    \label{phi_hk}
    \| \phi_{h,1} (s;t,x,\xi)\|_\infty  \le  |h|+C' \left|\int_t^s \|
    \phi_{h,1} (\tau ,t,x,\xi)\|_\infty d\tau \right|,
    \end{align}
 where $C'={\rm max}\{ 1,Cn\}$.
Since $|s-t|\le T$, applying Gronwall's inequality to \eqref{phi_hk}
 gives
    \begin{align}
    \label{phi-gronwall}
    \| \phi_{h,1} (s;t,x,\xi)\|_\infty \le |h|\, e^{C'|s-t|}\le |h|\,e^{C'T}.
    \end{align}
Hence, we obtain the desired result.
\end{proof}

Next, we show the differentiability of the solution to \eqref{ODE} in
initial datum.

\begin{lemmaA3}
Let $T>0$, $V\in C^{\infty} (\re\times \re^{n})$ satisfy
 \eqref{assumption_V} for all multi-indices $\alpha$ with $|\alpha|\ge
 2$. Suppose that $f(s;t,x,\xi)$ and $g(s;t,x,\xi)$ are solutions to \eqref{ODE}
 satisfying $f(t)=x$ and $g(t)=\xi$. Then $f(s;t,x,\xi)$ and
 $g(s;t,x,\xi)$ are $C^\infty$-function with respect to $x$ and $\xi$ in
 $|s-t|\le T$. 
\end{lemmaA3}

\begin{proof}
Let $A(s,y)$ be the $2n\times 2n$ matrix defined by  
    \begin{align}
    \label{henbun_eq_matrix}
    A(s,y) = \left(
    \begin{array}{cc}
    \mbox{\huge $O_n$} &    \mbox{$\begin{array}{ccc}
     -\dfrac{\partial^2}{\partial x_1 \partial x_1} V(s,y)& \cdots & 
     -\dfrac{\partial^2}{\partial x_1 \partial x_n} V(s,y)\\
     \vdots & \ddots & \vdots \\
         -\dfrac{\partial^2}{\partial x_n \partial x_1} V(s,y)& \cdots & 
     -\dfrac{\partial^2}{\partial x_n \partial x_n} V(s,y)\\
    \end{array}$}\vspace{2mm}\\
    \mbox{\huge $E_n$}
    & \mbox{\huge $O_n$}
    \end{array}
    \right) ,
    \end{align}
 where $O_n$ is the $n\times n$ zero matrix and $E_n$ is the $n\times n$
 identity matrix.
Let $h\in\re\backslash\{ 0\}$ and put 
    \begin{multline*}
    \phi_{h,j} (s;t,x,\xi)\\ = \big( f(s;t,x+he_j,\xi)-f(s;t,x,\xi),\, 
    g(s;t,x+he_j,\xi)-g(s;t,x,\xi)\big)
    \end{multline*}
 and 
    \begin{multline*}
    \psi_{h,j} (s;t,x,\xi)\\ =\big(f(s;t,x,\xi+he_j)-f(s;t,x,\xi),\, 
    g(s;t,x,\xi+he_j)-g(s;t,x,\xi)\big),
    \end{multline*}
 where $e_j=(0,\cdots ,0,\overset{\underset{\smallsmile}{j}}{1},0,
 \cdots,0)\in \re^{n}$ and $j=1,2,\ldots ,n$.
Suppose that
 $$w^{(k)}(s;t,x,\xi)=(w_{1,k}(s;t,x,\xi),\ldots ,w_{2n,k}(s;t,x,\xi))$$
 is the solution of 
    \begin{align}
    \label{w_ODE}
    \begin{cases}
    \dfrac{dw(s)}{ds}=w(s)A\big( s,f(s;t,x,\xi)\big), \\
    w(t)=(0,\cdots ,0,
    \overset{\underset{\smallsmile}{k}}{1},0,\cdots,0),
    \end{cases}
    \end{align}
 where $k=1,2,\ldots ,2n$.

First, we show that 
    \begin{align}
    \label{uniform-1}
    \lim_{h\to 0}\underset{|s-t|\le T}{\rm sup}\left\|
    \dfrac{\phi_{h,j}(s;t,x,\xi)}{h}-w^{(j)}(s;t,x,\xi)\right\|_{\infty}=0.
    \end{align}
From \eqref{ODE} and \eqref{vbibunnosa}, it is easy to see that 
    \begin{align}
    \label{phi-bibun}
    & \dfrac{d}{ds}\left( \dfrac{\phi_{h,j}
     (s;t,x,\xi)}{h}\right)\notag \\
    &=\dfrac{1}{h} \int_0^1 \phi_{h,j}(s,t,x,\xi) \notag\\
    &\qquad  \times A\big( s,f(s;t,x+he_j,\xi)
    +\theta (f(s;t,x,\xi)-f(s;t,x+he_j,\xi))
    \big) d\theta \notag\\ 
    &=\dfrac{1}{h}\phi_{h,j} (s,t,x,\xi) A\big( s,f(s;t,x,\xi)\big) 
      +\gamma_{h,j} (s,t,x,\xi),
    \end{align}
where 
    \begin{multline}
    \label{gamma}
     \gamma_{h,j} (s;t,x,\xi)=\dfrac{1}{h}
     \int_0^1 \phi_{h,j} (s,t,x,\xi) \bigg\{
     A\Big( s,f(s;t,x+he_j,\xi)\\+\theta
     \big( f(s;t,x,\xi)-f(s;t,x+he_j,\xi)\big)\Big)
    -A\big( s,f(s;t,x,\xi)\big)\bigg\} d\theta.
        \end{multline}
By the definition of $A(s,y)$ and Lemma A.1, there exists 
 $C>0$ such that 
    \begin{align}
    \label{Anosa}
    & \quad\bigg\|
     A\Big( s,f(s;t,x+he_j,\xi)+\theta
     \big( f(s;t,x,\xi)-f(s;t,x+he_j,\xi)\big)\Big)\notag\\
    &\qquad \qquad\qquad\qquad\qquad\qquad\qquad\qquad\qquad\qquad
     -A\big( s,f(s;t,x,\xi)\big)\bigg\|_\infty\notag\\
    &\le Cn (1-\theta ) \| f(s;t,x+he_j,\xi) - f(s,t,x,\xi)\|_\infty\notag\\
    &\le Cn \|\phi_{h,j}(s;t,x,\xi)\|_\infty
    \end{align}
 for  $\theta \in [0,1]$.
Thus \eqref{phi-gronwall}, \eqref{gamma} and \eqref{Anosa} yield 
    \begin{align}
    \label{gamma-norm}
    \|\gamma_{h,j} (s;t,x,\xi)\|_\infty
    &\le \dfrac{2Cn^2}{|h|} \|\phi_{h,j} (s;t,x,\xi) \|_\infty^2\notag\\
    &\le 2Cn^2 \|\phi_{h,j} (s;t,x,\xi)\|_\infty e^{C'|s-t|}
    \end{align}
 for $C'>0$.
As  $V$ satisfies the estimate \eqref{assumption_V} for all
 multi-indices $\alpha$ with $|\alpha| \ge 2$,
 there exists $M>0$ such that
    \begin{equation}
    \label{A_norm}
    \| A(s,f(s;t,x,\xi))\|_\infty \le M.
    \end{equation}
Since $\phi_{h,j} (t;t,x,\xi)/h-w^{(j)}(t;t,x,\xi)=0$, we have,
 by \eqref{w_ODE} and \eqref{phi-bibun}, 
    \begin{align*}
    &\quad\dfrac{\phi_{h,j} (s;t,x,\xi)}{h}-w^{(j)}(s;t,x,\xi)\\
    &=\int_t^s \left(\dfrac{1}{h}\dfrac{d}{d\tau}\phi_{h,j} (\tau;t,x,\xi)-
     \dfrac{d}{d\tau}w^{(j)}(\tau;t,x,\xi)\right)d\tau \\
    &=\int_t^s \gamma_{h,j}(\tau;t,x,\xi)\,d\tau \\
    &\qquad\qquad +
    \int_t^s 
    \left( \dfrac{\phi_{h,j} (\tau ;t,x,\xi)}{h}-
    w^{(j)} (\tau ,t,x,\xi)\right)A\big( \tau ,f(\tau ;t,x,\xi)\big)
     \,d\tau .
    \end{align*}
So we have, by \eqref{gamma-norm} and \eqref{A_norm},
    \begin{align}
    \label{phi-w}
    &\bigg\| \dfrac{\phi_{h,j} (s;t,x,\xi)}{h}
    -w^{(j)} (s;t,x,\xi) \bigg\|_\infty  \notag\\
    &\le  2Cn^2 e^{C'T}\int_{t-T}^{t+T} \|\phi_{h,j} (\tau;t,x,\xi)
     \|_\infty d\tau\notag\\
    &\qquad\qquad\qquad +2nM \left| \int_t^s
    \bigg\| \dfrac{\phi_{h,j} (\tau ;t,x,\xi)}{h}
    -w^{(j)} (\tau ,t,x,\xi)\bigg\|_\infty \,d\tau \right| 
    \end{align}
 for $|s-t|\le T$.
Applying Gronwall's inequality to \eqref{phi-w}, we obtain
    \begin{multline*}
    \bigg\| \dfrac{\phi_{h,j} (s ,t,x,\xi)}{h}-w^{(j)} (s;t,x,\xi)
     \bigg\|_\infty\\
    \le 2Cn^2 e^{(2nM+C')T}
      \int_{t-T}^{t+T}  \|\phi_{h,j} (\tau ,t,x,\xi)\|_\infty d\tau
    \end{multline*}
 for $|s-t|\le T$.
By Lemma A.2, we obtain \eqref{uniform-1}.
Thus, we have
    \begin{align*}
    \frac{\partial f_l (s;t,x,\xi)}{\partial x_k}=w_{l,k} (s;t,x,\xi)
    \in C(\re) 
    \end{align*}
and
    \begin{align*}
    \frac{\partial g_l (s;t,x,\xi)}{\partial x_k}=w_{n+l,k} (s;t,x,\xi)
    \in C(\re)
    \end{align*}
 for $k,l=1,\ldots ,n$.

On the other hand, in the similar calculation as above, we have
    \begin{align*}
    \lim_{h\to 0}\underset{|s-t|\le T}{\rm sup}\left\|
    \dfrac{\psi_{h,j}(s;t,x,\xi)}{h}-w^{(n+j)}(s;t,x,\xi)\right\|_{\infty}=0.
    \end{align*}
Thus,
    \begin{align*}
    \frac{\partial f_l (s;t,x,\xi)}{\partial \xi_k}=w_{l,n+k} (s;t,x,\xi)
    \in C(\re) 
    \end{align*}
 and 
    \begin{align*}
    \frac{\partial g_l (s;t,x,\xi)}{\partial \xi_k}=w_{n+l,n+k} (s;t,x,\xi)
    \in C(\re)
    \end{align*}
 for $k,l=1,\ldots ,n$. Hence, $f(s;t,x,\xi)$ and $g(s;t,x,\xi)$ are
 $C^1$-function with respect to $x$ and $\xi$.

Using above fact, we can easily show, by induction, that if $V(t,x)$ is
 $C^{r+1}$-function in $x$ then $f(s;t,x,\xi)$ and $g(s;t,x,\xi)$ are
 $C^r$-function with respect to $x$ and $\xi$.
Therefore we obtain the desired result.
\end{proof}

\noindent{\bf Proof of Lemma \ref{hensuuhenkan}.} 
Put
$w^{(k)}(s;t,x,\xi)=\left( w_{1,k}, w_{2,k}, \ldots ,w_{2n,k}\right)$
 for $1\le k\le 2n$. By Lemma A.3, $w^{(k)}(s)$ are the solutions of
    \begin{align*}
    \begin{cases}
    \dfrac{dw(s)}{ds}=w(s)A\big( s,f(s;t,x,\xi)\big), \\
    w(t)=(0,\cdots ,0,
    \overset{\underset{\smallsmile}{k}}{1},0,\cdots,0),
    \end{cases}
    \end{align*}
 where $A(s,f(s;t,x,\xi))=(a_{ij})$ is defined by
 \eqref{henbun_eq_matrix}. Then 
    \begin{multline}
    \label{lem31_1}
    \dfrac{d({\rm det}M(s;t,x,\xi))}{ds}\\[1mm]
    =\!\left|
    \begin{array}{ccc}
    \dfrac{d w_{1,1}}{ds} & \cdots&\dfrac{d w_{1,2n}}{ds} \\
    w_{2,1} &\cdots &  w_{2,2n} \\
    \vdots && \vdots\\
    w_{2n,1} &\cdots &  w_{2n,2n} \\
    \end{array}
    \right| +\cdots + \left|
    \begin{array}{ccc}
    w_{1,1} & \cdots& w_{1,2n} \\
    w_{2,1} & \cdots& w_{2,2n} \\
    \vdots && \vdots\\
    \dfrac{d w_{2n,1}}{ds}&\cdots&\dfrac{d w_{2n,2n}}{ds} \\
    \end{array}
 \right| .
\end{multline}
Since 
 $\frac{d w^{(k)}(s)}{d s}=w^{(k)}(s)A(s,f(s;t,x,\xi))$,
 we have
 $\frac{d w_{l,k}(s)}{d s}=\sum_{j=1}^{2n} a_{jl}w_{j,k}$
 and then 
    \begin{align}
    \label{lem31_2}
    \left|\begin{array}{ccc}
    \dfrac{dw_{1,1}}{ds} & \cdots & 
    \dfrac{dw_{1,2n}}{ds} \\
    w_{2,1} & \cdots &  w_{2,2n} \\
    \vdots && \vdots\\
    w_{2n,1} & \cdots &  w_{2n,2n} \\
    \end{array} \right|
    =\sum_{j=1}^{2n}a_{j1} \left|
    \begin{array}{ccc}
    w_{j,1} & \cdots &  w_{j,2n} \\
    w_{2,1} & \cdots &  w_{2,2n} \\
    \vdots && \vdots\\
    w_{2n,1} & \cdots &  w_{2n,2n} \\
    \end{array}
    \right|  =a_{11}{\rm det}M(s).
    \end{align}
From \eqref{lem31_1} and \eqref{lem31_2}, we have
    $$\dfrac{d({\rm det}M(s;t,x,\xi))}{ds}
    =(a_{11}+\cdots +a_{2n2n}){\rm det}M(s;t,x,\xi).$$
Since $\text{tr}A(s,f(s;t,x,\xi))=0$, we have 
$\frac{d ({\rm det}M(s;t,x,\xi))}{ds}=0.$
Therefore 
$$
{\rm det}M(s;t,x,\xi)={\rm det}M(t;t,x,\xi)={\rm det}E_{2n}=1.\quad
\square
$$


~\\

\end{document}